\documentclass[10pt]{amsart}%
\usepackage{graphicx}
\usepackage{amscd, color}
\usepackage{amsmath}
\usepackage{amsfonts}
\usepackage{amssymb}%
\providecommand{\U}[1]{\protect\rule{.1in}{.1in}}
\providecommand{\U}[1]{\protect\rule{.1in}{.1in}}
\providecommand{\U}[1]{\protect\rule{.1in}{.1in}} \textwidth 16.3cm
\theoremstyle{plain}

\newtheorem{theorem}{Theorem}[section]

\numberwithin{equation}{section}

\begin{document}
\title{A remark on the paper \textquotedblleft A Unified Pietsch Domination
Theorem\textquotedblright}
\author{Daniel Pellegrino and Joedson Santos}
\address[D. Pellegrino]{ Departamento de Matem\'{a}tica, Universidade Federal da
Para\'{\i}ba, 58.051-900 - Jo\~{a}o Pessoa, Brazil\\
[J. Santos] Departamento de Matem\'{a}tica, Universidade Federal de Sergipe,
49500-000- Itabaiana, Brazil}
\email{dmpellegrino@gmail.com}
\thanks{D. Pellegrino by INCT-Matem\'{a}tica, PROCAD-NF Capes, CNPq Grant
620108/2008-8 (Ed. Casadinho) and CNPq Grant 301237/2009-3. }

\begin{abstract}
In this short communication we show that the Unified Pietsch Domination proved
in \cite{BPRn} remains true even if we remove two of its apparently crucial hypothesis.

\end{abstract}
\maketitle

\section{Introduction}

Let $X$, $Y$ and $E$ be (arbitrary) non-void sets, $\mathcal{H}$ be a family
of mappings from $X$ to $Y$, $G$ be a Banach space and $K$ be a compact
Hausdorff topological space. Let
\[
R\colon K\times E\times G\longrightarrow\lbrack0,\infty)~\text{and}%
\mathrm{~}S\colon{\mathcal{H}}\times E\times G\longrightarrow\lbrack0,\infty)
\]
be arbitrary mappings.

A mapping $f\in\mathcal{H}$ is said to be $RS$-abstract $p$-summing if there
is a constant $C>0$ so that%
\begin{equation}
\left(  \sum_{j=1}^{m}S(f,x_{j},b_{j})^{p}\right)  ^{\frac{1}{p}}\leq
C\sup_{\varphi\in K}\left(  \sum_{j=1}^{m}R\left(  \varphi,x_{j},b_{j}\right)
^{p}\right)  ^{\frac{1}{p}}, \label{33M}%
\end{equation}
for all $x_{1},\ldots,x_{m}\in E,$ $b_{1},\ldots,b_{m}\in G$ and
$m\in\mathbb{N}$.

The main result of \cite{BPRn} proves that under certain assumptions on $R$
and $S$ there is a quite general Pietsch Domination-type Theorem. More
precisely $R$ and $S$ must satisfy the three properties below:

\noindent\textbf{(1)} For each $f\in\mathcal{H}$, there is a $x_{0}\in E$ such
that
\[
R(\varphi,x_{0},b)=S(f,x_{0},b)=0
\]
for every $\varphi\in K$ and $b\in G$.\newline\noindent\textbf{(2)} The
mapping
\[
R_{x,b}\colon K\longrightarrow\lbrack0,\infty)~\text{defined by}%
~R_{x,b}(\varphi)=R(\varphi,x,b)
\]
is continuous for every $x\in E$ and $b\in G$.\newline\noindent\textbf{(3)}
For every $\varphi\in K,x\in E,0\leq\eta\leq1,b\in G$ and $f\in{\mathcal{H}}$,
the following inequalities hold:
\[
R\left(  \varphi,x,\eta b\right)  \leq\eta R\left(  \varphi,x,b\right)
~\text{and}\mathrm{~}\eta S(f,x,b)\leq S(f,x,\eta b).
\]

The Pietsch Domination Theorem from \cite{BPRn} reads as follows:

\begin{theorem}
\label{yn} If $R$ and $S$ satisfy \textbf{(1)}, \textbf{(2)} and \textbf{(3)}
and $0<p<\infty,$ then $f\in{\mathcal{H}}$ is $RS$-abstract $p$-summing if and
only if there is a constant $C>0$ and a Borel probability measure $\mu$ on $K$
such that%
\begin{equation}
S(f,x,b)\leq C\left(  \int_{K}R\left(  \varphi,x,b\right)  ^{p}d\mu\right)
^{\frac{1}{p}}%
\end{equation}
for all $x\in E$ and $b\in G.$
\end{theorem}

The aim of this note is to show that, surprisingly, the hypothesis
\textbf{(1)} and \textbf{(3)} are not necessary. So, Theorem \ref{yn} is true
for arbitrary $S$ (no hypothesis is needed) and the map $R$ just needs to
satisfy \textbf{(2)}.

\section{A recent approach to PDT}

In a recent preprint \cite{joed} we have extended the Pietsch Domination
Theorem from \cite{BPRn} to a more abstract setting, which allows to deal with
more general nonlinear mappings in the cartesian product of Banach spaces. In
the present note we shall recall the argument used in \cite{joed} and a
combination of this argument with an interesting argument due to M. Mendel and
G. Schechtman (used in \cite{BPRn}) will help us to show that Theorem \ref{yn}
is valid without the hypothesis \textbf{(1)} and \textbf{(3)} on $R$ and $S$.

The first step is to prove Theorem \ref{yn} without the hypothesis
\textbf{(1)}. This result is proved in \cite{joed} in a more general setting.
Since the paper \cite{joed} is unpublished and we just need a very particular
case, we prefer to sketch the proof for this particular case. The proof of
this particular case is essentially Pietsch's original proof on a nonlinear disguise.

\begin{theorem}
\label{part}Suppose that $R$ and $S$ satisfy \textbf{(2)} and \textbf{(3)}. A
map $f\in{\mathcal{H}}$ is $RS$-abstract $p$-summing if and only if there is a
constant $C>0$ and a Borel probability measure $\mu$ on $K$ such that%
\begin{equation}
S(f,x,b)\leq C\left(  \int_{K}R\left(  \varphi,x,b\right)  ^{p}d\mu\right)
^{1/p} \label{gupdt}%
\end{equation}
for all $x\in E$ and $b\in G.$
\end{theorem}

\begin{proof}
If (\ref{gupdt}) holds it is easy to show that $f$ is $RS$-abstract
$p$-summing. For the converse, consider the (compact) set $P(K)$ of the
probability measures in $C(K)^{\ast}$ (endowed with the weak-star topology).
For each $(x_{j})_{j=1}^{m}$ in $E,$ $(b_{j})_{j=1}^{m}$ in $G$ and
$m\in\mathbb{N},$ let $g:P(K)\rightarrow\mathbb{R}$ be defined by%
\[
g\left(  \mu\right)  =\sum_{j=1}^{m}\left[  S(f,x_{j},b_{j})^{p}-C^{p}\int
_{K}R\left(  \varphi,x_{j},b_{j}\right)  ^{p}d\mu\right]
\]
and $\mathcal{F}$ be the set of all such $g.$ Using \textbf{(3)}, one can
prove that the family $\mathcal{F}$ is concave and each $g\in\mathcal{F}$ is
convex and continuous. Besides, for each $g\in\mathcal{F}$ there is a measure
$\mu_{g}\in P(K)$ such that $g(\mu_{g})\leq0.$ In fact, from \textbf{(2)}
there is a $\varphi_{0}\in K$ so that%
\[
\sum_{j=1}^{m}R\left(  \varphi_{0},x_{j},b_{j}\right)  ^{p}=\sup_{\varphi\in
K}\sum_{j=1}^{m}R\left(  \varphi,x_{j},b_{j}\right)  ^{p}%
\]
and, considering the Dirac measure $\mu_{g}=\delta_{\varphi_{0}},$ we have
$g(\mu_{g})\leq0.$ So, Ky Fan Lemma asserts that there exists a $\overline
{\mu}\in P(K)$ so that%
\[
g(\overline{\mu})\leq0
\]
for all $g\in\mathcal{F}$ and by choosing an arbitrary $g$ with $m=1$ the
proof is done.
\end{proof}

\section{The main result}

Note that if each $\lambda_{j}$ is a positive integer, by considering each
$x_{j}$ repeated $\lambda_{j}$ times in (\ref{33M}) one can easily see that
(\ref{33M}) is equivalent to%
\begin{equation}
\left(  \sum_{j=1}^{m}\lambda_{j}S(f,x_{j},b_{j})^{p}\right)  ^{\frac{1}{p}%
}\leq C\sup_{\varphi\in K}\left(  \sum_{j=1}^{m}\lambda_{j}R\left(
\varphi,x_{j},b_{j}\right)  ^{p}\right)  ^{\frac{1}{p}}\label{eeee}%
\end{equation}
for all $x_{1},\ldots,x_{m}\in E,$ $b_{1},\ldots,b_{m}\in G,$ positive
integers $\lambda_{j\text{ }}$and $m\in\mathbb{N}$. Then it is possible to
show that (\ref{33M}) holds for positive rationals and finally extend to
positive real numbers $\lambda_{j}$ (using an argument of density)$.$ The
essence of this argument appears in \cite{BPRn, FaJo} and is credited to M.
Mendel and G. Schechtman.

Now using (\ref{eeee}) and invoking Theorem \ref{part} we can prove a Pietsch
Domination-type theorem with no hypothesis on $S$ and just supposing that $R$
satisfies \textbf{(2)}:

\begin{theorem}
\label{tu} Suppose that $S$ is arbitrary and $R$ satisfies \textbf{(2)}. A map
$f\in{\mathcal{H}}$ is $RS$-abstract $p$-summing if and only if there is a
constant $C>0$ and a Borel probability measure $\mu$ on $K$ such that%
\begin{equation}
S(f,x,b)\leq C\left(  \int_{K}R\left(  \varphi,x,b\right)  ^{p}d\mu\right)
^{1/p} \label{yr}%
\end{equation}
for all $x\in E$ and $b\in G$.\newline

\begin{proof}
It is clear that if $f$ satisfies (\ref{yr}) then $f\in{\mathcal{H}}$ is
$RS$-abstract $p$-summing. Conversely, if $f\in{\mathcal{H}}$ is $RS$-abstract
$p$-summing, then%
\begin{equation}
\left(  \sum_{j=1}^{m}\lambda_{j}S(f,x_{j},b_{j})^{p}\right)  ^{\frac{1}{p}%
}\leq C\sup_{\varphi\in K}\left(  \sum_{j=1}^{m}\lambda_{j}R\left(
\varphi,x_{j},b_{j}\right)  ^{p}\right)  ^{\frac{1}{p}}\label{bbb}%
\end{equation}
for all $x_{1},\ldots,x_{m}\in E,$ $b_{1},\ldots,b_{m}\in G,$ $\lambda
_{1},...,\lambda_{m}\in\lbrack0,\infty)$ and $m\in\mathbb{N}$. Let
\[
E_{1}=E\times G\text{ \ \ and \ \ }G_{1}=\mathbb{K}%
\]
and define
\[
\overline{R}\colon K\times E_{1}\times G_{1}\longrightarrow\lbrack
0,\infty)\text{ \ \ and \ \ }\overline{S}\colon{\mathcal{H}}\times E_{1}\times
G_{1}\longrightarrow\lbrack0,\infty)
\]
by%
\[
\overline{R}(\varphi,(x,b),\lambda)=\left\vert \lambda\right\vert
R(\varphi,x,b)\text{ \ and \ }\overline{S}(f,(x,b),\lambda)=\left\vert
\lambda\right\vert S(f,x,b).
\]

From (\ref{bbb}) we conclude that
\[
\left(  \sum_{j=1}^{m}\overline{S}(f,(x_{j},b_{j}),\eta_{j})^{p}\right)
^{\frac{1}{p}}\leq C\sup_{\varphi\in K}\left(  \sum_{j=1}^{m}\overline
{R}\left(  \varphi,(x_{j},b_{j}),\eta_{j}\right)  ^{p}\right)  ^{\frac{1}{p}}%
\]
for all $x_{1},\ldots,x_{m}\in E,$ $b_{1},\ldots,b_{m}\in G,$ $\eta
_{1},...,\eta_{m}\in\mathbb{K}$ and $m\in\mathbb{N}$.

Since $\overline{R}$ and $\overline{S}$ satisfy \textbf{(2)} and \textbf{(3)},
from Theorem \ref{part} we conclude that there is a measure $\mu$ so that
\[
\overline{S}(f,(x,b),\eta)\leq C\left(  \int_{K}\overline{R}\left(
\varphi,(x,b),\eta\right)  ^{p}d\mu\right)  ^{1/p}%
\]
for all $x\in E,$ $b\in G$ and $\eta\in\mathbb{K}.$ Hence it easily follows
that, for all $x\in E$ and $b\in G,$ we have
\[
S(f,x,b)\leq C\left(  \int_{K}R(\varphi,x,b)^{p}d\mu\right)  ^{1/p}.
\]

\end{proof}
\end{theorem}

\end{document}